\documentclass[12pt,twoside,reqno]{amsart}
\newtheorem{thm}{Theorem}[section]
\newtheorem{lemma}{Lemma}[section]
\newtheorem{rem}{Remark}[section]

\theoremstyle{definition}

\theoremstyle{remark}

\newcommand{\R}{{\mathbb R}}
\newcommand{\N}{{\mathbb N}}

\newcommand{\noi}{\noindent}

\numberwithin{equation}{section}

\begin{document}
\title[nD Kuramoto-Sivashinsky]	
{	Global Regular Solutions for the multi-dimensional Kuramoto-Sivashinsky equation posed on smooth domains
}
\author{N. A. Larkin}

\address
{
	Departamento de Matem\'atica, Universidade Estadual
	de Maring\'a, Av. Colombo 5790: Ag\^encia UEM, 87020-900, Maring\'a, PR, Brazil
}

\bigskip

\email{ nlarkine@uem.br;nlarkine@yahoo.com.br }
\date{}

\begin{abstract}Initial- boundary value problems for the n-dimensional ($n$ is a natural number from the interval [2,7]) Kuramoto-Sivashinsky equation posed on  smooth bounded  domains in $\R^n$  were considered. The existence and uniqueness of  global  regular  solutions  as well as their exponential decay  have been established.
\end{abstract}

\maketitle

\section{Introduction}\label{introduction}
This work concerns the existence, uniqueness, regularity  and  exponential decay rates of solutions to initial-boundary value problems for the $n$-dimensional Kuramoto-Sivashinsky  equation (KS) 
\begin{align}
	& \phi_t+\Delta^2 \phi+ \Delta \phi +\frac{1}{2}|\nabla \phi|^2=0.
\end{align}
Here  $n$ is a natural number from the interval [2,7], $\Delta$ and $\nabla$ are the Laplacian and the gradient in $\R^n.$
In \cite{kuramoto}, Kuramoto studied the turbulent phase waves  and Sivashinsky in \cite{sivash} obtained an asymptotic equation which simulated the evolution of a disturbed plane flame front. See also \cite{Cross}.
Mathematical  results on  initial and initial- boundary value problems for one-dimensional (1.1)  are presented in \cite{Iorio, Guo,cousin,Larkin2,temam1,otto,temam2,zhang}, see references cited there for more information. In
\cite{gramchev,Guo}, the initial value problem for the multi-dimensional (1.1) type equations has been considered.
Two-dimensional periodic problems for the K-S- equation and its modifications posed on rectangles were studied in \cite{kukavica,molinet,temam1,sell,temam2}, where some results on the existence of weak solutions and nonlinear stability have been established. In \cite{Larkin}, initial-boundary value problems for the 3D Kuramoto-Sivashinsky equation have been studied; the existence, uniqueness and exponential decay of global regular solutions have been proved.\\ 
For $n$ dimensions, $ x=(x_1,...,x_n)$, $n=2,3,4,5,6,7$, (1.1) can be rewritten in the form of the following system:

 \begin{align}
	(u_j)_t+\Delta^2 u_j+\Delta u_j +\frac{1}{2}\sum _{i=1}^n(u_i)^2_{x_j}=0,\; j=1,...,n,&   \\
	(u_i)_{x_j}=(u_j)_{x_i},\; j\ne i,\;\; i,j=1,...,n.
	\end{align}
where $u_j=(\phi)_{x_j},\;j=1,...,n.$ Let $\Omega_n= \prod_{i=1}^n(0,L_i)$ be the minimal nD parallelepiped containing a given smooth domain $\Bar{D_n}.$
 First essential problem that arises while one studies either (1.1) or (1.2)-(1.3), are  destabilizing effects of 
$\Delta u_j;$ they may be damped by dissipative terms $\Delta^2 u_j$ provided $D_n$ has some specific properties. In order to understand this, we use Steklov`s inequalities to estimate

$$a\|u_j\|^2\leq \|\nabla u_j\|^2,\;a\|\nabla u_j\|^2\leq \|\Delta u_j\|^2;\;a=\sum_{i=1}^n\frac{\pi^2}{L_i^2},\; \;j=1,...,n.$$

A simple analysis shows that if
\begin{equation}
1-\frac{1}{a}>0,
\end{equation} then $\Delta^2 u_j$ damp $\Delta u_j$.  Naturally, here appear  admissible domains where (1.4) is fulfilled, so called "thin domains", 
where  some $L_i$ are   sufficiently small while others $L_j$ may be  large $ i,j=1,...,7; \;i\ne j.$

Second essential problem is presence of semi-linear terms in (1.2) which are interconnected. This does not allow to obtain the first estimate independent of $u_j$ and leads to a connection between $L_i$ and $u_j(0), \;i,j=1,...,7.$ \\
Our aim in this work is to study n-dimensional initial-boundary value problems for (1.2)-(1.3) posed on smooth domains, where the existence and uniqueness of global regular solutions as well as their exponential decay  of the $H^2(D_n)$-norm have been established. Moreover, we obtained a "smoothing effect"  for solutions with respect to initial data. Although, cases $n=2,3$ are not new, we included them for the sake of generality.

This work has the following structure: Section I is Introduction. Section 2 contains notations and auxiliary facts. In Section 3, formulation of an initial-boundary value problem  in a smooth bounded domain $D_n$ is given. The existence and uniqueness of global regular solutions, exponential decay of the $H^2(D_n)$-norm and a "smoothing effect" have been established.  Section 4 contains conclusions.

\section{Notations and Auxiliary Facts}

let $D_n$ be a suffuciently smooth domain in $\R^n,$ where $n$ is a fixed natural number, $n\in [2,7]$ satisfying
 the Cone condition, \cite{Adams}, and $x=(x_1,...,x_n) \in D_n.$ We use the standard notations of Sobolev spaces $W^{k,p}$, $L^p$ and $H^k$ for functions and the following notations for the norms \cite{Adams}
for scalar functions $f(x,t):$
$$\| f \|^2 = \int_{D_n} | f |^2dx, \hspace{1cm} \| f \|_{L^p(D_n)}^p = \int_{D_n} | f  |^p\, dx,$$
$$\| f \|_{W^{k,p}(D_n)}^p = \sum_{0 \leq \alpha \leq k} \|D^\alpha f \|_{L^p(D_n)}^p, \hspace{1cm} \| f \|_{H^k(D_n)} = \| f \|_{W^{k,2}(D_n)}.$$

When $p = 2$, $W^{k,p}(D_n) = H^k(D_n)$ is a Hilbert space with the scalar product 
$$((u,v))_{H^k(D_n)}=\sum_{|j|\leq k}(D^ju,D^jv),\;
\|u\|_{L^{\infty}(D_n)}=ess\; sup_{D_n}|u(x)|.$$
We use a notation $H_0^k(D_n)$ to represent the closure of $C_0^\infty(D_n)$, the set of all $C^\infty$ functions with compact support in $D_n$, with respect to the norm of $H^k(D_n)$.

\begin{lemma}[Steklov's Inequality \cite{steklov}] Let $v \in H^1_0(0,L).$ Then
	\begin{equation}\label{Estek}
	\frac {\pi^2}{L^2}\|v\|^2(t) \leq \|v_x\|^2(t).
	\end{equation}
\end{lemma}

\begin{lemma}
	[Differential form of the Gronwall Inequality]\label{gronwall} Let $I = [t_0,t_1]$. Suppose that functions $a,b:I\to \R$ are integrable and a function $a(t)$ may be of any sign. Let $u:I\to \R$ be a differentiable function satisfying
	\begin{equation}
		u_t (t) \leq a(t) u(t) + b(t),\text{ for }t \in I\text{ and } \,\, u(t_0) = u_0,
	\end{equation}
	then
	
	\begin{equation}u(t) \leq u_0 e^{ \int_{t_0}^t a(t)\, dt } + \int_{t_0}^t e^{\int_{t_0}^s a(r) \, dr} b(s)\, ds.\end{equation}.
\end{lemma}
\begin{proof}
	Multiply (2.3) by the integrating factor $e^{\int_{t_0}^{s} a(r)\, dr}$ and integrate from $t_0$ to $t$.
\end{proof}

The next Lemmas will be used in  estimates:

\begin{lemma}[See: \cite{friedman}, Theorem 9.1]  Let  $n$ be a natural number from the interval $[2,7];  \;D_n$ be a sufficiently smooth bounded domain in $\R^n$ satisfying the cone condition and $v \in H^4(D_n)\cap H^1_0(D_n)$.  then
for all $n$ defined above,	
	\begin{equation}
	\sup_{D_n}|v(x)|\leq C_n\| v\|_{H^4(D_n)}.
	\end{equation}
The constant $C_n$ depends on $n, D_n.$	
\end{lemma}

\begin{lemma}\label{lemma3}
	Let $f(t)$ be a continuous  positive function such that
	
	\begin{align} & f'(t) + (\alpha - k f^n(t)) f(t) \leq 0,\;t>0,\;n\in\N,\label{lemao1}\\
		& \alpha - k f^n(0)> 0,\;k>0.\label{lemao2}\end{align}
	\noi Then
	
	\begin{equation}f(t) < f(0)\end{equation}
	
	\noi for all $t > 0$.
\end{lemma}

\proof{  Obviously, $f'(0) + (\alpha - k f^n(0)) f^n(0) \leq 0$. Since $f$ is continuous, there exists $T>0$ such that $f(t)<f(0)$ for every $t \in [0,T)$. Suppose that  $f(0) = f(T)$. Integrating \eqref{lemao1}, we find
	
	$$f(T)+ \int_0^T(\alpha - k f^n(t)) f(t) \, dt \leq f(0).$$
	
	\noi Since	$$  \int_0^T(\alpha - k f^n(t)) f(t) \, dt >0,$$ then $f(T)<f(0).$ This  contradicts that $f(T)=f(0).$ Therefore, $f(t) < f(0)$ for all $t > 0.$ \\
	The proof of Lemma 2.5 is complete.  $\Box$

\section{K-S equation  posed on smooth domains}

 Let $\Omega_n$ be the minimal nD-parallelepiped containing a given bounded smooth domain $\Bar{D_n}\in\R^n,\;n=1,...,7$:
$$\Omega_n =\{ x\in \R^n; x_i\in (0,L_i)\},\; u_i=(\phi)_{x_i},\;i=1,...,n.$$ Fixing any natural $n=2,...,7,$  consider  in $Q_n=D_n\times (0,t)$  the following initial-boundary value problem:
 
 	\begin{align}
 		(u_j)_t+\Delta^2 u_j+\Delta u_j +\frac{1}{2}\sum _{i=1}^n(u^2_i)_{x_j}=0,\; j=1,...,n,&\\
 		(u_i)_{x_j}=(u_j)_{x_i},\; j\ne i,\;\; i,j=1,...,n;&\\
   	u_j|_{\partial D_n}=\Delta u_j|_{\partial D_n}=0,\; t>0,&\\
 	u_j(x,0)=u_{j0}(x),\;j=1,...,n,\;\;x \in D_n.
\end{align}
 \begin{lemma}
 Let $f\in H^4(D_n)\cap H^1_0(D_n)$ and $ \Delta f|_{\partial D_n}=0.$ Then 
 $$
 a\|f\|^2 \leq	\|\nabla f\|^2,\;\;a^2\|f\|^2\leq \|\Delta f\|^2,\;\;a\|\nabla f\|^2\leq \|\Delta f\|^2,$$
 $$	a^2\|\Delta f\|^2\leq \|\Delta^2 f\|^2,\;\; \|\Delta \nabla f\|^2\leq \|\Delta^2 f\|\|\Delta f\|\leq \frac{1}{a}\|\Delta^2 f\|^2.$$

where $ a=\sum_{i=1}^n\frac{\pi^2 }{L^2_i} ,\;\;\|f\|^2=\int_{D_n}f^2(x)dx.$
\end{lemma}
\begin{proof} We have
	$$\|\nabla f\|^2=\sum_{i=1}^n\|f_{x_i}\|^2.$$
	Define 
\begin{equation}\tilde{f}(x,t)= \begin{cases}f(x,t) \text{ if } x \in D_n;\\
		0 \text{ if } x \in \Omega_n/D_n.\end{cases} \end{equation}
	Making use of Steklov`s inequalities for $\tilde{f}(x,t)$ and taking into account that $\|\nabla f\|=\|\nabla \tilde{f}\|,$\;we get
	\begin{align*}
		\|\nabla f\|^2\geq a\|f\|^2,\;
	\text{where} \; a=\sum_{i=1}^n\frac{\pi^2 }{L_i^2} .
	\end{align*}	
   	On the other hand,
	$$a\|f\|^2\leq \|\nabla f\|^2 =-\int_{D_n}f\Delta fdx\leq \|\Delta f\|\|f\|.$$
	This implies
	$$a\|f\|\leq \|\Delta f\|\;\;\text{and} \;\;a^2\|f\|^2\leq \|\Delta f\|^2.$$
	Consequently, \;$a\|\nabla f\|^2\leq \|\Delta f\|^2.$ 	Similarly,
	$$ \|\Delta f\|^2=\int_{D_n} f\Delta^2 fdx\leq \|\Delta^2 f\|\|f\|\leq \frac{1}{a}\|\Delta^2 f\|\|\Delta f\|.$$
	Hence, \;\;$a\|\Delta f\|\leq \|\Delta^2 f\|.$ Moreover, $$ \|\Delta \nabla f\|^2=-\int_{D_n} \Delta^2 f \Delta fdx\leq \|\Delta^2 f\|\|\Delta f\|\leq \frac{1}{a}\|\Delta^2 f\|^2.$$
	Proof of Lemma 3.1 is complete.
	\end{proof}
\begin{rem} Assertions of Lemma 3.1 are true if the function $f$ is replaced respectively by $u_j,\;j=1,...,n$.
	\end{rem}

\begin{lemma} In conditiions of Lemma 3.1,
	\begin{align}\|f\|^2(t)_{H^2(D_n)}\leq 3\|\Delta f\|^2(t), &\\
\|f\|^2(t)_{H^4(D_n)}\leq 5\|\Delta^2f\|^2(t), &\\
\sup_{D_n}|f(x)|\leq C_s\| \Delta^2 f\|,\; \text{where}\;C_s=5C_n.
\end{align}
\end{lemma}
\begin{proof} To prove (3.7), making use of Lemma 3.1, we find
$$\|f\|^2_{H^4(D_n)}=\|f\|^2+\|\nabla f\|^2+\|\Delta f\|^2+\|\nabla\Delta f\|^2+\|\Delta^2 f\|^2$$$$
\leq\Big(\frac{1}{a^4}+	\frac{1}{a^3}+\frac{1}{a^2}+\frac{1}{a}+1\Big)\|\Delta^2 f\|^2.$$
Since $a>1$, then  (3.7) follows. Similarly, (3.6) can be proved. Moreover, taking into account Lemma 2.3, we get (3.8).
\end{proof}
\begin{thm}	[Special basis]\label{existencethm}
Let  $n$ be a natural number from the interval [2,7]; $D_n\in\R^n$ be a bounded smooth domain satisfying the Cone condition and $\Omega_n$ be a minimal $nD$-parallelepiped containing $\Bar{D_n}.$ Let
\begin{equation}
\theta=	1-\frac{1}{a}=1-\frac{1}{\sum_{i=1}^n\frac{\pi^2 }{L^2_i} }>0.
\end{equation}
Given $$u_{j0}(D_n)\in H^2(D_n)\cap H^1_0(D_n),\;j=1,...,n$$ such that

\begin{align}
\theta-\frac{2 C_s^2 7^3}{a\theta}\Big(\sum_{j=1}^n\|\Delta u_j\|^2(0)\Big)>0,
\end{align}
then there exists a unique global regular solution to (3.1)- (3.4):
$$ u_j\in L^{\infty}(\R^+; H^2(D_n))\cap L^2(\R^+;H^4(D_n)\cap H^1_0(D_n));$$$$u_{jt}\in L^2(\R^+;L^2(D_n)), \;j=1,...,n.$$ Moreover,
\begin{align}\sum_{j=1}^n \|\Delta u_j\|^2(t)
		\leq \Big(\sum_{j=1}^n\|\Delta u_{j0}\|^2\Big)\exp\{-a^2t\theta/2\}
	\end{align}
and
\begin{align}
	\sum_{i=1}^n\|\Delta u_i\|^2(t)+\int_0^t\sum_{i=1}^n\|\Delta^2 u_i\|^2(\tau)d\tau \leq C\sum_{i=1}^n\|\Delta u_{i0}\|^2, \;t>0.\notag
\end{align}

\end{thm}
\begin{rem} In Theorem 3.1, there are two types of restrictions: the first one is pure geometrical,
	$$	1-\frac{1}{a}>0$$
	which is needed to eliminate  destabilizing effects of the terms $\Delta u_j$ in (3.1):
	
$$ \|\Delta u_j\|^2-\|\nabla u_j\|^2.$$	
	
It is clear that$$\lim_{L_i\to 0}a=\sum_{i=1}^n\frac{\pi^2 }{L^2_i}=+\infty,$$  hence to achieve  (3.9), it is possible  to  decrease $L_i,\;i=1,...,n$ allowing other $L_j,\;j\ne i$ to grow. \\
A situation with  condition	(3.10)
is more complicated: if initial data are not small, then it is possible either to decrease $L_i,\;i=1,...,n,$ to fulfill this condition or for fixed $L_i,\;i=1,...,n$ to decrease initial data $\|u_{j0}\|.$\\
\end{rem}	
	
\begin{proof} It is possible to construct Galerkin`s approximations to (3.1)-(3.4) by the following way:
let $w_j(x)$ be eigenfunctions of the problem:
$$\Delta^2 w_j-\lambda_jw_j=0\;\text{in}\; D_n;\;\;w_j|_{\partial D_n}=\Delta w_j|_{\partial D_n}=0, j=1,2,....$$ 
Define
\begin{align*}
u^N_j(x,t)=\sum_{k=1}^N g_k^j(t)w_j(x).
\end{align*}	
Unknown functions $g_i^j(t)$\;satisfy the following initial problems:
\begin{align*}
 (\frac{d}{dt}u^N_j,w_j)(t)+(\Delta^2 u^N_j,w_j)(t)+(\Delta u^N_j,w_j)(t)&\\
 +\frac{1}{2}(\sum_{i=1}^n(u^N_i)^2_{x_j},w_j)(t)=0,\\
g_k^j(0)=g_{0k}^j,\;\;j=1,...,n,\;\;k=1,2,... .	
\end{align*}	
	The estimates that follow may be established on Galerkin's approximations (see: \cite{Guo,cousin}), but it is more explicitly to prove them on smooth solutions of (3.1)-(3.4).

\noi {\bf Estimate I:} $u \in L^\infty(\R^+; H^2(D_n)\cap H^1_0(D_n))\cap L^2(\R^+;H^4(D_n)\cap H^1_0(D_n))$. 

For any natural $n\in  [2,7],$ multiply (3.1) by $2\Delta^2 u_j$ to obtain
\begin{align}
\frac{d}{dt} \|\Delta u_j\|^2(t)+2\|\Delta^2 u_j\|^2(t)+2\|\Delta^2 u\|(t)\|\Delta u_j\|(t)&\notag\\
+2\sum_{i=1}^n (u_i(u_i)_{x_j},\Delta^2 u_j)(t)=0.
\end{align}

Making use of (3.7) and Lemmas 2.3, 3.1, 3.2, we can write 

\begin{align}
	\frac{d}{dt} \|\Delta u_j\|^2(t)+2\theta\|\Delta^2 u_j\|^2(t)&\notag\\
	\leq 2\Big[\sum_{i=1}^n \sup_{D_n}| u_i(x,t)|\|\nabla u_i\|(t)\Big]\|\Delta^2 u_j\|(t)&\notag\\
	\leq
2\Big[C_s\sum_{i=1}^n\|\Delta^2 u_i\|(t)\|\nabla u_i\|(t)\Big]\|\Delta^2 u_j\|(t); j=1,...,n.
\end{align}
Summing over $j=1,...,n$ and making use of Lemma 3.1, we  rewrite (3.12) in the form:

\begin{align*}
	\frac {d}{dt}\sum_{j=1}^n\|\Delta u_j\|^2(t)+2\theta \sum_{j=1}^n \|\Delta^2 u_j\|(t)&\notag\\
\leq 2C_s n\Big(\sum_{j=1}^n\|\nabla  u_j\|(t)\Big)\Big[\sum_{j=1}^n \|\Delta^2 u_j\|^2(t)\Big]&\notag\\
\leq \Big[\frac{\theta}{2}+\frac{2C_s^2 n^2}{\theta}\Big(\sum_{j=1}^n\|\nabla u_j\|(t)\Big)^2\Big]\sum_{j=1}^n\|\Delta^2 u_j\|^2(t)&\notag\\
\leq \Big[\frac{\theta}{2}+\frac{2C_s^2 n^3}{\theta}\Big(\sum_{j=1}^n\|\nabla u_j\|^2(t)\Big)\Big]\sum_{j=1}^n\|\Delta^2 u_j\|^2(t)&\notag\\
\leq \Big[\frac{\theta}{2}+\frac{2C_s^2 n^3}{a\theta}\Big(\sum_{j=1}^n\|\Delta u_j\|^2(t)\Big)\Big]\sum_{j=1}^n\|\Delta^2 u_j\|^2(t).
\end{align*}
Taking this into account, we transform (3.12) in the form

\begin{align}
	\frac {d}{dt}\sum_{j=1}^n\|\Delta u_j\|^2(t)+\frac{\theta}{2} \sum_{j=1}^n \|\Delta^2 u_j\|^2(t)&\notag\\
	+ \Big[\theta-\frac{2C_s^2 n^3}{a\theta}\Big(\sum_{j=1}^n\|\Delta u_j\|^2(t)\Big)\Big]\sum_{j=1}^n\|\Delta^2 u_j\|^2(t)\leq 0.
\end{align}

 Condition (3.10) and Lemma 2.4 guarantee that
 
 $$\theta-\frac{2C_s^2 n^3}{a\theta}\Big(\sum_{j=1}^n\|\Delta u_j\|^2(t)\Big)>0,\; \;t>0.$$
 
 Hence (3.13) can be rewritten as
 \begin{align}
 	\frac {d}{dt}\sum_{j=1}^n\|\Delta u_j\|^2(t)+\frac{a^2\theta}{2}\sum_{j=1}^n \|\Delta u_j\|^2(t)\leq 0.
 \end{align}
 Integrating, we get
 
 \begin{align}\sum_{i=1}^n \|\Delta u_j\|^2(t)	\leq\sum_{j=1}^n\|\Delta u_{j0}\|^2\exp\{-a^2\theta t/2\}.
 \end{align}
 
Returning to (3.22), we find
\begin{align}
\sum_{i=1}^n\|\Delta u_i\|^2(t)+\int_0^t\sum_{i=1}^n\|\Delta^2 u_i\|^2(\tau)d\tau \leq C\sum_{i=1}^n\|\Delta u_{i0}\|^2.
\end{align}

Finally, directly from (3.1), we obtain $$(u_j)_t \in L^2(\R^+;L^2(D_n)),\;j=1,...,n.$$

This completes the proof of the existence part of Theorem 3.1.

\begin{lemma} A regular solution of Theorem 3.1 is uniquely defined.
\end{lemma}
\begin{proof} Let $u_j$ and $v_j$,\;$j=1,...,n$ be two distinct solutions to (3.1)-(3.4). Denoting $w_j=u_j-v_j$,	we come to the following system:
\begin{align}
(w_j)_t+\Delta^2 w_j+\Delta w_j +\frac{1}{2}\sum_{i=1}^n\Big(u^2_i-v^2_i\Big)_{x_j} =0,&\\	
(w_j)_{x_i}=(w_i)_{x_j},\;i\ne j,&\\
w_j|_{\partial D_n}=\Delta w_j|_{\partial D_n}
=0,\; t>0;&\\
w_j(x,0)=0 \;\text{in} \;D,\;\;j=1,...,n.
\end{align}
Multiply (3.18) by $2w_j$ to obtain

\begin{align} \frac{d}{dt}\|w_j\|^2(t)+2\|\Delta w_j\|^2(t)-2\|\nabla w_j\|^2(t)&\notag\\
	-\sum_{i=1}^n\Big(\{u_i+v_i\}w_i,(w_j)_{x_j}\Big)(t)=0,\;j=1,...,n.
	\end{align}
Making use of Lemmas 2.3 and 3.1, 3.2, we estimate
$$
	 I=\sum_{i=1}^n(\{u_i+v_i\}w_i,(w_j)_{x_j})\leq\frac{\epsilon}{2}\|\nabla w_j\|^2+\frac{1}{2\epsilon}\Big(\sum_{i=1}^n\|\{u_i+v_i\}w_i\|\Big)^2$$
$$\leq \frac{\epsilon}{2a}\|\Delta w_j\|^2	+\frac{2}{\epsilon}\sum_{i=1}^n \sup_{D_n}\{u^2_i(x,t)_i+v^2_i(x,t)\}\|w^2_i\|(t)$$
$$\leq \frac{\epsilon}{2a}\|\Delta w_j\|^2	+\frac{4C_s^2}{\epsilon}\Big[\sum_{i=1}^n\{\|\Delta^2 u_i\|^2+\|\Delta^2 v_i\|^2\}\Big]\sum_{j=1}^n\|w_j\|^2.$$ Here $\epsilon$ is an arbitrary positive number. Substituting $I$ into (3.22), we get

\begin{align}
\frac{d}{dt}\|w_j\|^2(t)+(2-\frac{2}{a}-\frac{\epsilon}{2a})\|\Delta w_j\|^2(t)&\notag\\
\leq \frac{4C_s^2}{\epsilon}\Big[\sum_{i=1}^n\{\|\Delta^2 u_i\|^2+\|\Delta^2 v_i\|^2\}\Big]\sum_{j=1}^n\|w_j\|^2.
\end{align}
Taking $\epsilon=\frac{\theta}{2}$ and summing up over $j=1,...,n,$  we transform (3.23) as follows:
\begin{align*}\frac{d}{dt}\sum_{j=1}^n\|w_j\|^2(t)\leq C \Big[\sum_{i=1}^n\{\|\Delta^2 u_i\|^2&\notag\\+\|\Delta^2 v_i\|^2+ \|u_i\|^2(t)+\|v_i\|^2(t)\}\Big]\sum_{j=1}^n\|w_j\|^2,\;i=1,..., n.
\end{align*}

Since by (3.17) and Lemma 3.1,$$\|\Delta^2 u_i\|^2(t), \|\Delta^2 v_i \|^2(t)\in L^1(\R^+)$$ and
$$\|u_i\|(t),\|v_i\|^2(t)\in L^1(\R^+),\;i=1,...,n,$$ then by  Lemma 2.2, 

$$ \|w_j\|^2(t)\equiv 0\;\;j=1,...,n,\; \text{for all}\;\; t>0.$$
Hence
\begin{align*}
 u_j(x,t)\equiv v_j(x,t);j=1,...,n.
 \end{align*}
The proof of Lemma 3.3, and consequently, Theorem 3.1 are complete.

\end{proof}

\end{proof}

\

\section{ Conclusions}

This work concerns formulation and solvability of initial-boundary value problems for the n-dimensional ($n$ is a natural number from the interval [2,7],) Kuramoto-Sivashinsky system (3.1)-(3.2) posed on smooth bounded domains. Theorem 3.1 contains  results on existence and uniqueness of global regular solutions as well as exponential decay of the $H^2(D_n)$-norm, where $D_n$ is a smooth bounded  domain in $\R^n$.  We define a set of admissible domains, where    destabilizing effects of terms $\Delta u_j$ are damped by dissipativity of $\Delta^2 u_j$ due to condition (3.9). This set contains "thin domains", see \cite{Iftimie,molinet,sell}, where some dimensions of $D_n$ are small while others may be  large. Since initial- boundary value problems studied in this work do not admit the first a priori estimate independent of $t,u_j,$ in order to prove the existence of global regular solutions, we put conditions (3.10) connecting geometrical properties of $D_n$ with initial data $u_{j0}$. 
Moreover, Theorem 3.1 provides "smoothing effect": initial data $u_{j0}\in H^2(D_n)\cap H^1_0(D_n)$ imply that
$ u_j\in L^2(\R^;H^4(D_n)).$\\

\section*{Conflict of Interests}

The author declares that there is no conflict of interest regarding the publication of this paper.

\medskip

\bibliographystyle{torresmo}

\end{document}